\documentclass[12pt,reqno]{article}

\usepackage[T1]{fontenc}
\usepackage[sc]{mathpazo}
\usepackage{microtype}
\usepackage{amsmath}
\usepackage{amssymb}
\usepackage{amsthm}
\usepackage{bm}
\usepackage{dsfont}
\usepackage{authblk}
\usepackage{fullpage}
\usepackage{float}
\usepackage{comment}
\usepackage{tikz}
\usepackage{mathtools}
\usepackage[shortlabels]{enumitem}
\usepackage{complexity}
\usepackage{hyperref}
\usepackage[capitalize, nameinlink]{cleveref}
\usepackage{nag}
\usepackage{xcolor}
\usepackage{colortbl}
\usepackage{systeme}

\newcommand{\declarecolor}[2]{\definecolor{#1}{RGB}{#2}\expandafter\newcommand\csname #1\endcsname[1]{\textcolor{#1}{##1}}}
\declarecolor{White}{255, 255, 255}
\declarecolor{Black}{0, 0, 0}
\declarecolor{LightGray}{216, 216, 216}
\declarecolor{Gray}{127, 127, 127}
\declarecolor{Orange}{237, 125, 49}
\declarecolor{LightOrange}{251,229, 214}
\declarecolor{Yellow}{255, 192, 0}
\declarecolor{LightYellow}{255, 242, 200}
\declarecolor{Blue}{91, 155, 213}
\declarecolor{LightBlue}{222, 235, 247}
\declarecolor{Green}{112, 173, 71}
\declarecolor{LightGreen}{226, 240, 217}
\declarecolor{Navy}{68, 114, 196}
\declarecolor{LightNavy}{218, 227, 243}

\hypersetup{
	colorlinks=true,
	pdfpagemode=UseNone,
	citecolor=Green,
	linkcolor=Navy,
	urlcolor=Black,
	pdfstartview=FitW
}

\newcommand{\declareperson}[1]{\expandafter\newcommand\csname#1\endcsname[1]{\textcolor{Orange}{#1: ##1}}}

\declareperson{Bobby}

\theoremstyle{plain}
\newtheorem{theorem}{Theorem}[section]
\newtheorem{lemma}[theorem]{Lemma}
\newtheorem{corollary}[theorem]{Corollary}
\newtheorem{proposition}[theorem]{Proposition}

\newtheorem*{no-lemma}{Lemma}

\theoremstyle{definition}
\newtheorem{definition}[theorem]{Definition}

\theoremstyle{remark}
\newtheorem{remark}[theorem]{Remark}
\newtheorem{example}[theorem]{Example}

\newlist{parts}{enumerate}{10}
\setlist[parts]{label=\arabic*.,ref=\arabic*}
\makeatletter
\if@cref@capitalise
	\crefname{partsi}{Part}{Parts}
\else
	\crefname{partsi}{part}{parts}
\fi
\makeatother
\Crefname{partsi}{Part}{Parts}

\newlist{forms}{enumerate}{10}
\setlist[forms]{label=\roman*),ref=(\roman*)}
\makeatletter
\if@cref@capitalise
	\crefname{formsi}{}{}
\else
	\crefname{formsi}{}{}
\fi
\makeatother
\Crefname{formsi}{}{}

\newlist{conds}{enumerate}{10}
\setlist[conds]{label=\rm{(\arabic*)},ref=\arabic*}
\makeatletter
\if@cref@capitalise
	\crefname{condsi}{}{}
\else
	\crefname{condsi}{}{}
\fi
\makeatother
\Crefname{condsi}{}{}

\usetikzlibrary{shapes, patterns, decorations, fit, intersections}

\DeclareMathOperator{\sgn}{sgn}

\DeclareMathOperator{\cter}{center}

\DeclarePairedDelimiterX{\card}[1]{\lvert}{\rvert}{#1}
\DeclarePairedDelimiterX{\abs}[1]{\lvert}{\rvert}{#1}
\DeclarePairedDelimiterX{\norm}[1]{\lVert}{\rVert}{#1}
\DeclarePairedDelimiterX{\tuple}[1]{\lparen}{\rparen}{#1}
\DeclarePairedDelimiterX{\parens}[1]{\lparen}{\rparen}{#1}
\DeclarePairedDelimiterX{\brackets}[1]{\lbrack}{\rbrack}{#1}
\DeclarePairedDelimiterX{\set}[1]\{\}{#1}
\let\Pr\relax
\DeclarePairedDelimiterXPP{\Pr}[1]{\mathbb{P}}[]{}{#1}
\DeclarePairedDelimiterXPP{\PrX}[2]{\mathbb{P}_{#1}}[]{}{#2}
\DeclarePairedDelimiterXPP{\Ex}[1]{\mathbb{E}}[]{}{#1}
\DeclarePairedDelimiterXPP{\ExX}[2]{\mathbb{E}_{#1}}[]{}{#2}

\title{The Sixth Moment of Random Determinants}
\author{Dominik Beck}
\affil{\small Charles University in Prague}
\author{Zelin Lv}
\affil{\small The University of Chicago}
\author{Aaron Potechin}
\affil{\small The University of Chicago}

\begin{document}
	\maketitle
\begin{abstract}
In this paper, we determine the sixth moment of the determinant of an asymmetric $n \times n$ random matrix where the entries are drawn independently from an arbitrary distribution $\Omega$ with mean $0$. Furthermore, we derive the asymptotic behavior of the sixth moment of the determinant as the size of the matrix tends to infinity. %Our results extend previous work on the moments of determinants of random matrices and have potential applications in various fields such as statistical physics and random matrix theory.
\end{abstract}

Acknowledgement: This research was supported by NSF grant CCF-2008920.
\section{Introduction}
The behavior of the determinant of a random matrix has been extensively studied. One line of work analyzed the kth moment of a random determinant, i.e., the expected value of the kth power of the determinant of a random matrix. Tur\'{a}n observed that the second moment of the determinant of an $n \times n$ matrix (where the entries have mean $0$ and variance $1$) is $n!$. The fourth moment of the determinant was determined by Nyquist, Rice, and Riordan \cite{NRR54}. For the special case when the entries of the random matrix are Gaussian, several works \cite{forsythe1952extent, NRR54} showed that the kth moment of the determinant is $\prod_{j=0}^{\frac{k}{2}-1}{\frac{(n + 2j)!}{(2j)!}}$.

Several variations of this question have also been analyzed. Dembo \cite{10.2307/43637554} generalized these results for $p \times n$ matrices (where we consider $E\left[\det (MM^T)^{\frac{k}{2}}\right]$ rather than $E\left[\det (M)^k\right]$). Very recently, Beck \cite{https://doi.org/10.48550/arxiv.2207.09311} determined the fourth moment of the determinant of an $n \times n$ random matrix with independent entries from an arbitrary distribution (which may not have mean $0$). For symmetric random matrices, the second moment of the determinant was analyzed by Zhurbenko \cite{Zhu68}.

A second line of work analyzed the distribution of the determinant of a random matrix. Girko \cite{girko1980central, doi:10.1137/S0040585X97975939} showed that under some assumptions on the entries of the random matrix, the logarithm of the determinant obeys a central limit theorem, but his proofs are quite difficult. Nguyen and Vu \cite{10.1214/12-AOP791} gave a simpler proof of a stronger version of this central limit theorem.

A third line of work analyzed the probability that a random $n \times n$ matrix with $\pm{1}$ entries is singular. Koml\'{o}s \cite{komlos1967determinant, komlos1968determinant} was the first to prove that this probability is $o(1)$. Kahn, Koml{\'o}s, and Szemer{\'e}di \cite{kahn1995probability} proved that this probability is at most $.999^n$, which was the first exponential upper bound. A series of works \cite{https://doi.org/10.1002/rsa.20109, tao2007singularity,BOURGAIN2010559} improved this upper bound culminating in the work of Tikhomirov \cite{10.4007/annals.2020.191.2.6} who proved an upper bound of $(\frac{1}{2} + o(1))^n$, which is tight. For the symmetric case, Costello, Tao, and Vu \cite{CTV06} showed that the probability that a random symmetric matrix with $\pm{1}$ entries is singular is $o(1)$ and further work has improved this bound.

For a recent survey of results on random matrices, see \cite{Vu20}. In this paper, we extend the first line of work by analyzing the sixth moment of the determinant of a random matrix.
\begin{comment}
Previous draft:
For many problems, it is important to understand the behavior of their higher moments. In this paper, we seek an expression for the sixth moment of determinants of random matrices. In our setting, each element is independently and identically drawn from an arbitrary distribution $\Omega$ with mean $0$ and variance $1$. 
% In order to understand the behavior of the singularity of matrices, we take a look on the higher moments of random determinants
In this paper, we present a new way to find the higher moments of random determinants by inclusion-exclusion principle and make a connection from this problem to counting a certain kind of graph.
\end{comment}
\subsection{The Fourth Moment of Random Determinants}
Before stating our main result, we describe the results of Nyquist, Rice, and Riordan \cite{NRR54} on the fourth moment of the determinant of a random matrix.
\begin{definition}
Given a distribution $\Omega$, we define $\mathcal{M}_{n \times n}(\Omega)$ to be the distribution of $n \times n$ matrices where each entry is drawn independently from $\Omega$.
\end{definition}
\begin{definition}
Given a distribution $\Omega$, we define $m_k$ to be the kth moment of $\Omega$, i.e., 
\[
m_k = E_{x \sim \Omega}[x^k].
\]
\end{definition}
\begin{definition}
We define $f_{k}(n) = E_{M \sim \mathcal{M}_{n \times n}(\Omega)}\left[\det (M)^{k}\right]$ to be the expected value of the $k$-th power of the determinant of a random $n\times n$ matrix. Similarly, we define $p_{k}(n)$ to be the expected value of the $k$-th power of the permanent of a random $n\times n$ matrix.
\end{definition}
\begin{remark}
Both $f_{k}(n)$ and $p_{k}(n)$ depend on the moments of $\Omega$, but we write $f_{k}(n)$ and $p_{k}(n)$ rather than $f_{k,\Omega}(n)$ and $p_{k,\Omega}(n)$ for brevity.
\end{remark}
Nyquist, Rice, and Riordan \cite{NRR54} showed that $f_4(n) = n!y_{n}$ where $y_n$ obeys the recurrence relation
\[
y_n = (n + m_4 - 1)y_{n-1} + (3-m_4)(n-1)y_{n-2}.
\]
where $y_0 = 1$ and $y_1 = m_4$. They further observed that if we take the generating function $Y(t) = \sum_{t = 0}^{\infty}{\frac{y_{n}t^{n}}{n!}}$ then $Y(t) = (1-t)^{-3}e^{(m_4-3)t}$. From this generating function, they found the equation
\[
f_4(n) = n!y_n =\frac{(n!)^2}{2}\sum^n_{k=0}\frac{(n-k+1)(n-k+2)}{k!}(m_4 - 3)^k.
\]
To prove their results, Nyquist, Rice, and Riordan \cite{NRR54} counted $4 \times n$ tables with certain properties. As we describe in Section \ref{sec:preliminaries}, we use the same general approach though our analysis is considerably more intricate.
\subsection{Our Results}
Our main results are as follows.
% \begin{theorem}\label{main_thm}
% For any distribution $\Omega$ such that $m_1 = m_3 = 0$ and $m_2 = 1$,
% \begin{align*}
%     f_6(n)&=\sum_{j=0}^{n}{\sum_{a=0}^{j}\binom{n}{j}\binom{j}{a}(m_6 - 15)^{a} (m_4 - 3)^{(j-a)}D_{n,a,j-a}}.\\
%     D_{n,a,b}&=\left(\prod_{j=0}^{a+b-1}{n-j}\right)\left(\sum_{i=0}^{b}\binom{b}{i}C_{i}H_{n,b-i,a,b}\right)P_{n-a-b}.\\
%     P_{n}&=n(n+2)(n+4)P_{n-1} \text{ where } P_{0} = 1 \text{. Equivalently, } P_{n} = \frac{n!(n+2)!(n+4)!}{2!4!}.\\
%     C_{n}&=(n-1)(C_{n-1}+15C_{n-2}) \text{ where } C_0 = 1 \text{ and } C_1 = 0.\\
%     H_{n,j,a,b}&=\sum_{x=1}^{j}{\frac{\binom{j-1}{x-1}}{x!}j!\prod_{y=0}^{x-1}(3(n-a-b)-y).}
% \end{align*}
% \end{theorem}

\begin{theorem}\label{genf_thm}
For any distribution $\Omega$ such that $m_1 = m_3 = 0$ and $m_2 = 1$, the formal generating function $F_6(t) = \sum_{n=0}^\infty\frac{t^n}{(n!)^2} f_6(n)$ for $f_6(n)$ is
%we got the following formal generating function
\begin{equation*}
F_6(t) = \frac{e^{t \left(m_6-15 m_4 + 30\right)}}{48 \left(1+3t - m_4 t\right)\!{}^{15}} \sum _{i=0}^{\infty }
   \frac{(1+i) (2+i) (4+i)! t^i}{\left(1+3t -m_4 t\right){}^{3 i}}.
\end{equation*}
\end{theorem}

\begin{remark}
Performing Taylor expansion of this generating function, we get the formula for computing the sixth moment of random determinants, namely:
\end{remark}

\begin{corollary}\label{main_cor}
For any distribution $\Omega$ such that $m_1 = m_3 = 0$ and $m_2 = 1$,
\begin{align*}
f_6(n)=(n!)^2 \sum_{j=0}^n \sum_{i=0}^j\frac{(1\!+\!i) (2\!+\!i) (4\!+\!i)! }{48 (n-j)!}\binom{14\!+\!j\!+\!2i}{j-i} \left(m_6-15m_4+30\right)^{n-j} \left(m_4-3\right)^{j-i}.
\end{align*}
\end{corollary}
\begin{remark}
If $m_2 \neq 1$ then we can scale the distribution $\Omega$ by $\frac{1}{\sqrt{m_2}}$ (which changes the determinant of matrices in $\mathcal{M}_{n \times n}(\Omega)$ by a factor of $\left(\frac{1}{\sqrt{m_2}}\right)^{n}$) and then apply the result in Corollary \ref{main_cor}.
\end{remark}
\begin{remark}
If $\Omega = N(0,1)$ then $m_4 = 3$ and $m_6 = 15$ so $f_6(n) = P_n = \frac{n!(n+2)!(n+4)!}{48}$, which is a special case of the result that $f_k(n) = \prod_{j=0}^{\frac{k}{2}-1}{\frac{(n + 2j)!}{(2j)!}}$ when $\Omega = N(0,1)$ and $k$ is even.
\end{remark}
Another generalization is when $m_3 \neq 0$.
\begin{theorem}\label{m3nonzero_thm}
For any distribution $\Omega$ such that $m_1 = 0$ and $m_2 = 1$, the formal generating function $F_6(t) = \sum_{n=0}^\infty\frac{t^n}{(n!)^2} f_6(n)$ for $f_6(n)$ is
\begin{align*}
F_6(t) =\left(1+m_3^2t\right)\!{}^{10}\, \frac{e^{t \left(m_6-10 m_3^2-15m_4 + 30\right)}}{48 \left(1+3t-m_4t\right)^{15}} \sum _{i=0}^{\infty }
   \frac{(1+i) (2+i) (4+i)!t^i}{\left(1+3t-m_4t\right){}^{3 i}}.
\end{align*}
\end{theorem}
\begin{corollary}
For any distribution $\Omega$ such that $m_1 = 0$ and $m_2 = 1$,
\begin{align*}
f_6(n) = (n!)^2 \sum _{j=0}^n \sum _{i=0}^j \sum_{k=0}^{n-j} \frac{(1+i) (2+i) (4+i)! }{48 (n-j-k)!}\binom{10}{k}\binom{14+j+2i}{j-i} q_6^{n-j-k} q_4^{j-i}q_3^k,
\end{align*}
where
\begin{equation*}
q_6 = m_6-10 m_3^2-15m_4 + 30, \qquad\qquad q_4 = m_4-3, \qquad\qquad q_3 = m_3^2.
\end{equation*}
\end{corollary}

Below, we show the values of $f_k(n)$ and $p_k(n)$ when $\Omega = \{-1,1\}$ for small values of $k$ and $n$. We note that when $\Omega = \{-1,1\}$, $f_4(n)$ is the integer sequence A052127 in the On-Line Encyclopedia of Integer Sequences \cite{oeis_f4}. In the entry for this integer sequence, it is noted that $f_4(n) \sim (n!)^2\frac{(n^2+7n+10)}{(2e^2)}$ as $n \to \infty$.
\begin{table}[H]
\centering
\begin{tabular}{|l|l|l|l|l|l|l|}
\hline
n & $f_2(n)$ & $f_4(n)$ & $f_6(n)$ & $p_2(n)$ & $p_4(n)$ & $p_6(n)$  \\ \hline
1 & 1        & 1        & 1        & 1        & 1        & 1         \\ \hline
2 & 2        & 8        & 32       & 2        & 8        & 32        \\ \hline
3 & 6        & 96       & 1536     & 6        & 96       & 2976      \\ \hline
4 & 24       & 2112     & 282624   & 24       & 2112     & 513024    \\ \hline
5 & 120      & 68160    & 66846720 & 120      & 68160    & 157854720 \\ \hline
\end{tabular}
\caption{Values of $f_k(n)$ and $p_k(n)$ for $\Omega = \{-1,1\}$, $k \leq 6$, and $n \leq 5$}
\label{values-det-per}
\end{table}
\begin{remark}
    Note that for all $n \in \mathbb{N}$, $p_2(n) = f_2(n)$ and $p_4(n) = f_4(n)$. However, for $n \geq 3$, $p_6(n) > f_6(n)$.
\end{remark}
We can describe the asymptotic behavior of $f_6$ using the following asymptotic expansion.
\begin{theorem}\label{asym_thm}
For all $R \in \mathbb{N} \cup \{0\}$,
\[
f_6(n) = \frac{e^{3q_4}(n!)^2}{48} \left(\sum_{k=0}^R c_k (n+6-k)!\right) {\pm}O\left((n!)^2(n+6-R-1)!\right),
\]
where the coefficients $c_k$ are the Taylor expansion coefficients of the function $C(t) = \sum_{k\geq 0} c_k t^k$,
\begin{align*}
C(t) = \, & e^{(q_6-3 q_4^2)t+ q_4^3 t^2 } \left(1+q_3 t\right)^{10} \bigg{(} 1-2 \left(3 q_4+4\right) t \\
& +3 \left(5 q_4^2+8 q_4+4\right) t^2-4 \left(q_4^2 \left(5 q_4+6\right)\right) t^3+q_4^3 \left(15 q_4+8\right) t^4-6
   q_4^5 t^5+q_4^6 t^6 \bigg{)}.
\end{align*}
\end{theorem}

\begin{remark}
For the first terms in the expansion, we have
\begin{align*}
f_6(n) & \sim \frac{e^{3 m_4-9}}{48} (n!)^3 \bigg{(}n^6 + \left(m_6-3 m_4^2-3 m_4+34\right) n^5 +\frac{1}{2} \left(m_6^2-10 m_3^4\right.\\
& \left. +9 m_4^4+20 m_4^3-183 m_4^2-126
   m_4-6 m_4^2 m_6-6 m_4 m_6+56 m_6+905\right) n^4 + \cdots \bigg{)}    
\end{align*}    
\end{remark}

\begin{remark}
Note that when $\Omega = \{-1,1\}$, as $n \to \infty$,
\[
f_6(n) \sim \frac{(n!)^3}{48
   e^6}\left(n^6+29 n^5+335 n^4+\frac{5861 n^3}{3}+\frac{17944 n^2}{3}+\frac{44036 n}{5}+\frac{167536}{45}-\frac{210176}{63 n}\right).
\]
\end{remark}

\section{Preliminaries}\label{sec:preliminaries}
To prove Theorem \ref{genf_thm}, we need a few definitions and a key lemma.
\begin{definition}
Given natural numbers $k$ and $n$ where $k$ is even, we define an even $k \times n$ table to be a $k \times n$ table where each row is a permutation of $[n]$ and each column contains each number an even number of times. We define $T_{k,n}$ to be the set of all even $k \times n$ tables.
\end{definition}
% \textcolor{red}{Formal definition and example for pairs(done)}
%\begin{definition}
%Let's call a 
%We say that a table is \textit{even} if every element in each column appears an even number of times. Define $T_{k,n}$ to be the set that contains all even tables of size $k \times n$.
%\end{definition}
% \begin{definition}
% Given an even table $t$ of size $k\times n$, we define its sign $\sgn(t)$ to be the product of the signs of its rows, which are permutations of $[n]$.
% \end{definition}
\begin{definition}
Given an even table $t$ of size $k\times n$, we define its sign $\sgn(t)$ to be the product of the signs of its rows, which are permutations of $[n]$.
\end{definition}
\begin{definition}
Given a column $c$ where each element is in $[n]$, we define its weight $w(c)$ to be 
\[
w(c) = \prod_{j=1}^{n}{m_{\# \text{ of times } j \text{ appears in column c}}}
.\]
For even $6 \times n$ tables, we say that a column is a 6-column if it contains some number $6$ times, a 4-column if it contains one number four times and another number two times, and a 2-column if it contains three different numbers two times. Observe that the weight of a 6-column is $m_6$, the weight of a 4-column is $m_4$, and the weight of a 2-column is $m_2$.
\end{definition}
\begin{definition}
Given an even $k\times n$ table $t$, we define its weight $w(t)$ to be the product of the weights of its columns.
%Given an even table $t$ of size $k\times n$, we define its weight $w(t)$ to be the product of the weights of its columns. If it's a 6-column, it's weight is $m_6$. If it's a 4-column, it's weight is $m_4$. And it's one if is a 2-column. Here $m_k$ denotes the $k-$moment of each entry.
\end{definition}
We can use the following proposition to reduce the problem of finding the sixth moment of a random determinant to a combinatorial problem.
\begin{proposition}\label{prop:tables}
For all even $k \in \mathbb{N}$, $f_k(n)=\sum_{t\in T_k(n)}\sgn(t)w(t)$ and 
$p_k(n)=\sum_{t\in T_k(n)}w(t)$.
\end{proposition}
\begin{proof} 
We observe that 
\[
f_k(n) = E_{A \sim \mathcal{M}_{n \times n}(\Omega)}\left[\sum_{\pi_1,\pi_2,...,\pi_k\in S_n}{\left(\prod_{i=1}^{k}\sgn (\pi_i)\right)\prod_{p=1}^{n}{ \left(\prod_{q=1}^{k}A_{p,\pi_q(p)}\right)}}\right]
\]
and 
\[
p_k(n) = E_{A \sim \mathcal{M}_{n \times n}(\Omega)}\left[\sum_{\pi_1,\pi_2,...,\pi_k\in S_n}{\prod_{p=1}^{n}{ \left(\prod_{q=1}^{k}A_{p,\pi_q(p)}\right)}}\right]
.\]
For each $p \in [n]$, we have that $E_{A \sim \mathcal{M}_{n \times n}(\Omega)}[\prod_{q=1}^{k}A_{p,\pi_q(p)}] = w(p)$ (i.e., the weight of column $p$), so $f_k(n)=\sum_{t\in T_k(n)}\sgn(t)w(t)$ and $p_k(n)=\sum_{t\in T_k(n)}w(t)$, as needed.
%First we note that for random entries with unit for even moments, the sum of all the possible determinants to the $k$-th power can be computed as: 
%\begin{align*}
%    2^{n^2}f_k(n)&=\sum_{\pi_1,\pi_2,...,\pi_k\in S_n}\left(\prod_{i=1}^{k}\sgn (\pi_i)\right)\sum_{i,j}\sum_{A_{i,j}=\pm 1}\prod_{p=1}^n (\prod_{q=1}^{k}A_{p,\pi_q(p)})
%\end{align*}
%By symmetry, we get a nonzero contribution to the final value if and only if the product $\prod_{p=1}^n (\prod_{q=1}^{k}A_{p,\pi_q(p)})$ is a perfect square, which means that each entry appears an even number of times in $\pi_1,\pi_2,..,\pi_k$. So we have
%\begin{align*}
%    2^{n^2}f_k(n)&=\sum_{t\in T_{k,n}}\sgn(t)\sum_{i,j}\sum_{A_{i,j}=\pm 1}=2^{n^2}\sum_{t\in T_{k,n}}\sgn(t)\\
%    f_k(n)&=\sum_{t\in T_{k,n}}\sgn(t)
%\end{align*}
%So for arbitrary distribution, we have $$ f_k(n)=\sum_{t\in T_{k,n}}\sgn(t)w(t)$$
\end{proof}
Thus, computing the kth moment of a random determinant is equivalent to summing the signed weights of all even tables of size $k\times n$.
\begin{corollary}
If $\Omega$ is the uniform Bernoulli distribution (i.e., the uniform distribution on $\{-1,1\}$) then $f_k(n)=\sum_{t\in T_{k,n}}\sgn(t)$ and $p_k(n)=|T_{k,n}|$.
\end{corollary}
\begin{corollary}
If $k=2$, $k=4$, or $n \leq 2$ then $p_k(n)=f_k(n)$. If $n \geq 3$, $k \geq 6$, and $k$ is even then $p_k(n)>f_k(n)$.
\end{corollary}
To analyze $f_6(n)$, it is useful to consider tables together with pairings of identical elements in each column.
\begin{definition}
Given an even $k \times n$ table $t$, we define a pairing $P$ on $t$ to be a set of matchings $\{M_i: i \in [n]\}$, one for each column, where each matching $M_i$ pairs up identical elements of column $i$. We define $\mathcal{P}(t)$ to be the set of all pairings on $t$.
\end{definition}
%The idea we used is to first count the possible \textit{pairings} in each table. Recall that all elements appear an even numbers of times in each column. So we call two identical elements a pair. To count this, we classify each column as $k-column$, where $k$ is the largest number of an element appears in this column. Once we have count the number of pairs, we use Inclusion-Exclusion to find the determinant. Note that \textit{pairings} also have signs.
%\begin{definition}
%Given an even table of size $2k\times n$, we can associate it with a \textit{pairing} by partitioning each of the column into $k$ $2-$tuples, where each of the tuples has two identical elements. Let $P_{k,n}$ denote the number of possible \textit{pairings} of all tables in $T_{k,n}$.
%\end{definition}
\begin{example} The table on the left below is an even $6\times 4$ table with $27$ possible \textit{pairings}. The table on the right shows one of the $27$ possible parings. 
\begin{table}[H]
\centering
\begin{tabular}{cc}
\begin{tabular}{|l|l|l|l|}
\hline
1 & 2 & 4 & 3 \\ \hline
1 & 2 & 4 & 3 \\ \hline
1 & 3 & 4 & 2 \\ \hline
1 & 3 & 4 & 2 \\ \hline
2 & 4 & 1 & 3 \\ \hline
2 & 4 & 1 & 3 \\ \hline
\end{tabular}&

\begin{tabular}{|l|l|l|l|}
\hline
\rowcolor[HTML]{FD6864} 
1 & {\color[HTML]{000000} 2} & 4 & \cellcolor[HTML]{FFFE65}3 \\ \hline
\rowcolor[HTML]{FD6864} 
1 & {\color[HTML]{000000} 2} & \cellcolor[HTML]{FFFE65}{\color[HTML]{000000} 4} & {\color[HTML]{000000} 3} \\ \hline
\rowcolor[HTML]{FFFE65} 
\cellcolor[HTML]{FCFF2F}{\color[HTML]{000000} 1} & 3 & {\color[HTML]{000000} 4} & \cellcolor[HTML]{67FD9A}{\color[HTML]{000000} 2} \\ \hline
\cellcolor[HTML]{FCFF2F}{\color[HTML]{000000} 1} & \cellcolor[HTML]{FFFE65}3 & \cellcolor[HTML]{FD6864}4 & \cellcolor[HTML]{67FD9A}{\color[HTML]{000000} 2} \\ \hline
\rowcolor[HTML]{67FD9A} 
2 & 4 & 1 & \cellcolor[HTML]{FD6864}{\color[HTML]{000000} 3} \\ \hline
\rowcolor[HTML]{67FD9A} 
2 & 4 & 1 & \cellcolor[HTML]{FFFE65}3 \\ \hline
\end{tabular}
\end{tabular}
\end{table}
\end{example}
\begin{proposition}
For each even $6 \times n$ table $t$, 
$$
|\mathcal{P}(t)| = 15^{\# \mathrm{\ of\ 6-columns\ in\ t}}3^{\# \mathrm{\ of\ 4-columns\ in\ t}}
.$$
\end{proposition}
%For simplicity, let $P_n$ denote $P_{6,n}$, and $T_n$ denote $T_{6,n}$.
\begin{definition}
We define $P_n = \sum_{t \in T_{k,n}}{\sgn(t)|\mathcal{P}(t)|}$.
\end{definition}
It turns out that $P_n$ can be easily computed and this is crucial for our results.
\begin{lemma}\label{pairing}
For all $n \in \mathbb{N}$, $P_{n}=n(n+2)(n+4)P_{n-1}$ where $P_0=1$.
\end{lemma}
This lemma follows from the fact that when $\Omega = N(0,1)$ and $k$ is even, the kth moment of the determinant is $\prod_{j=0}^{\frac{k}{2}-1}{\frac{(n + 2j)!}{(2j)!}}$. We give a direct proof of this lemma in Appendix \ref{sec:directproof}.

Note that $P_n = \sum_{t \in T_{k,n}}{\sgn(t)15^{\# \text{ of 6-columns in t}}3^{\# \text{ of 4-columns in t}}}$ while \\
$f_6(n) = \sum_{t \in T_{k,n}}{\sgn(t){m_6}^{\# \text{ of 6-columns in t}}{m_4}^{\# \text{ of 4-columns in t}}}$. If $\Omega = N(0,1)$ (or we at least have that $m_2 = 1$, $m_4 = 3$, and $m_6 = 15$) then $f_6(n) = P_n$. 
%In Section \ref{sec:maintheoremproof}, 
In the next section, we show how to handle other distributions $\Omega$ using inclusion/exclusion.
\section{Proof of Theorem \ref{genf_thm}}\label{sec:maintheoremproof}
Before preceding to the proof of Theorem \ref{genf_thm}, we first prove the following result on the sixth moment of random determinants. 
\begin{lemma} \label{main_lemma}
    For any distribution $\Omega$ such that $m_1 = m_3 = 0$ and $m_2 = 1$,
\begin{align*}
    f_6(n)&=\sum_{j=0}^{n}{\sum_{a=0}^{j}\binom{n}{j}\binom{j}{a}(m_6 - 15)^{a} (m_4 - 3)^{(j-a)}D_{n,a,j-a}}.\\
    D_{n,a,b}&=\left(\prod_{j=0}^{a+b-1}{(n-j)}\right)\left(\sum_{i=0}^{b}\binom{b}{i}C_{i}H_{n,b-i,a,b}\right)P_{n-a-b}.\\
    P_{n}&=n(n+2)(n+4)P_{n-1} \text{ where } P_{0} = 1 \text{. Equivalently, } P_{n} = \frac{n!(n+2)!(n+4)!}{2!4!}.\\
    C_{n}&=(n-1)(C_{n-1}+15C_{n-2}) \text{ where } C_0 = 1 \text{ and } C_1 = 0.\\
    H_{n,j,a,b}&=\sum_{x=1}^{j}{\frac{\binom{j-1}{x-1}}{x!}j!\prod_{y=0}^{x-1}(3(n-a-b)-y).}
\end{align*}
\end{lemma}

% Note: If we want to count the net number of tables with x 6-columns and y 4-columns, this is 
% $F(x,y,n) = \sum_{j=x+y}^{n}{\sum_{a=x}^{j-y}\binom{n}{j}\binom{j}{a}\binom{a}{x}\binom{j-a}{y}(-15)^{a-x} (-3)^{j-a-y}D_{n,a,j-a}}$
\begin{proof}
The idea behind the proof is as follows. We consider the tables where we know that some set $A \subseteq [n]$ of elements appear six times in a 6-column and another set $B \subseteq [n] \setminus A$ of elements appear four times in a 4-column and two times in a different column. We do not know whether the elements in $[n] \setminus (A \cup B)$ appear six times in a 6-column, appear four times in a 4-column and two times in a different column, or appear two times in three different columns. We consider these tables together with pairings for the columns which are unaccounted for by $A$ and $B$ (i.e., the columns which don't contain six of the same element in $A$ or four of the same element in $B$). 

To obtain $f_6(n)$, we compute the contribution from each $A$ and $B$ and then take an appropriate linear combination of these contributions so that the contribution from each individual table $t$ is $\sgn(t)w(t)$.
\begin{definition}
Given $A \subseteq [n]$ and $B \subseteq [n] \setminus A$, we define $D_{n,A,B}$ to be the set of tables in $T_{6,n}$ such that the elements in $A$ appear six times in a 6-column and the elements in $B$ appear four times in a 4-column and two times in a different column.

For each $t \in D_{n,A,B}$, we define  $\mathcal{P}_{A,B}(t)$ to be the set of pairings on $t$ where we exclude the 6-columns which contain six of the same element in $A$ and the 4-columns which include four of the same element in $B$.
\end{definition}
By symmetry, the contribution from each $D_{n,A,B}$ only depends on $|A|$ and $|B|$.
\begin{definition}
Given $n,a,b \in \mathbb{N} \cup \{0\}$ such that $a+b \leq n$, we define $D_{n,a,b}$ to be 
\[
D_{n,a,b} = \sum_{t \in D_{n,A,B}}{\sgn(t)|\mathcal{P}_{A,B}(t)|}
\]
where $A \subseteq [n]$, $B \subseteq [n] \setminus A$, $|A| = a$, and $|B| = b$.
\end{definition}
\begin{lemma}
For all $n \in \mathbb{N} \cup \{0\}$, $f_6(n) = \sum_{j=0}^{n}{\sum_{a=0}^{j}\binom{n}{j}\binom{j}{a}(m_6 - 15)^{a}(m_4 - 3)^{(j-a)}D_{n,a,j-a}}.$
\end{lemma}
\begin{proof}
Observe that 
\begin{align*}
&\sum_{j=0}^{n}{\sum_{a=0}^{j}\binom{n}{j}\binom{j}{a}(m_6 - 15)^{a}(m_4 - 3)^{(j-a)}D_{n,a,j-a}} \\
&=\sum_{A \subseteq [n]}{\sum_{B \subseteq [n] \setminus A}{\sum_{t \in D_{n,A,B}}{(m_6 - 15)^{|A|}(m_4 - 3)^{|B|}\sgn(t)|\mathcal{P}_{A,B}(t)|}}}.
\end{align*}
Given a table $t \in T_{6,n}$, let $A'$ be the set of elements in $[n]$ which appear six times in a 6-column of $t$ and let $B'$ be the set of element which appear four times in a 4-column of $t$. Now consider the contribution from $t$ in 
\[
\sum_{A \subseteq [n]}{\sum_{B \subseteq [n] \setminus A}{\sum_{t \in D_{n,A,B}}{(m_6 - 15)^{|A|}(m_4 - 3)^{|B|}\sgn(t)|\mathcal{P}_{A,B}(t)|}}}.
\]
We have that whenever $A \subseteq A'$ and $B \subseteq B'$, $t \in D_{n,A,B}$ and $|\mathcal{P}_{A,B}(t)| = 15^{|A' \setminus A|}3^{|B' \setminus B|}$. Thus, the contribution from $t$ is
\[
\sum_{A \subseteq A'}{\sum_{B \subseteq B'}(m_6 - 15)^{|A|}(m_4 - 3)^{|B|}15^{|A' \setminus A|}3^{|B' \setminus B|}\sgn(t)} = {m_6}^{|A'|}{m_4}^{|B'|}\sgn(t) = \sgn(t)w(t).
\]
This implies that 
\[
\sum_{j=0}^{n}{\sum_{a=0}^{j}\binom{n}{j}\binom{j}{a}(m_6 - 15)^{a} (m_4 - 3)^{(j-a)}D_{n,a,j-a}} = \sum_{t \in T_{6,n}}{\sgn(t)w(t)} = f_6(n)
\]
as needed.
\end{proof}
We now compute $D_{n,a,b}$.
\begin{lemma}
For all $n,a,b \in \mathbb{N} \cup \{0\}$ such that $a+b \leq n$,
\[
    D_{n,a,b}=\left(\prod_{j=0}^{a+b-1}{(n-j)}\right)\left(\sum_{i=0}^{b}\binom{b}{i}C_{i}H_{n,b-i,a,b}\right)P_{n-a-b}
\]
where $C_n$ is given by the recurrence relation $C_{n}=(n-1)(C_{n-1}+15C_{n-2})$, $C_0 = 1$, $C_1 = 0$ and 
\[
    H_{n,j,a,b}=\sum_{x=1}^{j}{\frac{\binom{j-1}{x-1}}{x!}j!\prod_{y=0}^{x-1}(3(n-a-b)-y)}.
\]
\end{lemma}
\begin{proof}
To prove this lemma, we group the tables in $D_{n,A,B}$ based on the structure of the 4-columns containing four of the same element of $B$.
\begin{definition}
Given a table $t \in D_{n,A,B}$, we define the directed graph $G(t)$ to be the graph with vertices $B \cup \{center\}$ and the following edges. For each $b \in B$, 
\begin{enumerate}
    \item If there is a $b' \in B \setminus \{b\}$ such that there are two $b$ in the 4-column containing four $b'$ then we add an edge from $b$ to $b'$.
    \item If there is no such $b'$ then we add an edge from $b$ to $\cter$.
\end{enumerate}
\end{definition}
\begin{proposition}
For all $t \in D_{n,A,B}$, $G(t)$ has the following properties.
\begin{enumerate}
    \item For all $b \in B$, $\deg^+(b)=1$ and $\deg^-(b) \leq 1$.
    \item $\deg^+(\cter) = 0$.
\end{enumerate}
\end{proposition}
\begin{corollary}
For all $t \in D_{n,A,B}$, $G(t)$ consists of directed cycles and paths which end at $\cter$, all of which are disjoint except for their common endpoint.
\end{corollary}
\begin{example} This example shows the correspondence between a table $t$ and $G(t)$.
%\textit{cycle graph}.
\begin{figure}[H]
\centering
  \begin{minipage}[b]{0.4\textwidth} \centering
\begin{tabular}{|l|l|l|l|l|}
\hline
1 & 2 & 3 & 4 & 5 \\ \hline
1 & 2 & 3 & 4 & 5 \\ \hline
1 & 2 & 5 & 3 & 4 \\ \hline
1 & 2 & 5 & 3 & 4 \\ \hline
2 & 1 & 3 & 4 & 5 \\ \hline
2 & 1 & 3 & 4 & 5 \\ \hline
\end{tabular}
\caption{A $6\times 5$ table $t \in D_{5,\emptyset,[5]}$}
\end{minipage}
  \hfill
    \begin{minipage}[b]{0.4\textwidth} \centering

\tikzset{every picture/.style={line width=0.75pt}} %set default line width to 0.75pt        

\begin{tikzpicture}[x=0.75pt,y=0.75pt,yscale=-0.5,xscale=0.5]

%Shape: Circle [id:dp8592516726379484] 
\draw   (173,149) .. controls (173,135.19) and (184.19,124) .. (198,124) .. controls (211.81,124) and (223,135.19) .. (223,149) .. controls (223,162.81) and (211.81,174) .. (198,174) .. controls (184.19,174) and (173,162.81) .. (173,149) -- cycle ;
%Shape: Circle [id:dp7929164664132613] 
\draw   (42,149) .. controls (42,135.19) and (53.19,124) .. (67,124) .. controls (80.81,124) and (92,135.19) .. (92,149) .. controls (92,162.81) and (80.81,174) .. (67,174) .. controls (53.19,174) and (42,162.81) .. (42,149) -- cycle ;
%Curve Lines [id:da16990709544075422] 
\draw    (67,124) .. controls (96.7,84.79) and (165.6,87.33) .. (197.06,122.91) ;
\draw [shift={(198,124)}, rotate = 229.74] [color={rgb, 255:red, 0; green, 0; blue, 0 }  ][line width=0.75]    (10.93,-3.29) .. controls (6.95,-1.4) and (3.31,-0.3) .. (0,0) .. controls (3.31,0.3) and (6.95,1.4) .. (10.93,3.29)   ;
%Curve Lines [id:da4233037641731543] 
\draw    (198,174) .. controls (164.34,213) and (97.36,215.35) .. (67.88,175.23) ;
\draw [shift={(67,174)}, rotate = 54.98] [color={rgb, 255:red, 0; green, 0; blue, 0 }  ][line width=0.75]    (10.93,-3.29) .. controls (6.95,-1.4) and (3.31,-0.3) .. (0,0) .. controls (3.31,0.3) and (6.95,1.4) .. (10.93,3.29)   ;
%Shape: Circle [id:dp46417909078186304] 
\draw   (444,188) .. controls (444,174.19) and (455.19,163) .. (469,163) .. controls (482.81,163) and (494,174.19) .. (494,188) .. controls (494,201.81) and (482.81,213) .. (469,213) .. controls (455.19,213) and (444,201.81) .. (444,188) -- cycle ;
%Shape: Circle [id:dp392026598941698] 
\draw   (313,188) .. controls (313,174.19) and (324.19,163) .. (338,163) .. controls (351.81,163) and (363,174.19) .. (363,188) .. controls (363,201.81) and (351.81,213) .. (338,213) .. controls (324.19,213) and (313,201.81) .. (313,188) -- cycle ;
%Shape: Circle [id:dp7870823810243959] 
\draw   (376,92) .. controls (376,78.19) and (387.19,67) .. (401,67) .. controls (414.81,67) and (426,78.19) .. (426,92) .. controls (426,105.81) and (414.81,117) .. (401,117) .. controls (387.19,117) and (376,105.81) .. (376,92) -- cycle ;
%Straight Lines [id:da20304805290140426] 
\draw    (376,92) -- (338.94,161.24) ;
\draw [shift={(338,163)}, rotate = 298.16] [color={rgb, 255:red, 0; green, 0; blue, 0 }  ][line width=0.75]    (10.93,-3.29) .. controls (6.95,-1.4) and (3.31,-0.3) .. (0,0) .. controls (3.31,0.3) and (6.95,1.4) .. (10.93,3.29)   ;
%Straight Lines [id:da05069575110160529] 
\draw    (469,163) -- (427.04,93.71) ;
\draw [shift={(426,92)}, rotate = 58.8] [color={rgb, 255:red, 0; green, 0; blue, 0 }  ][line width=0.75]    (10.93,-3.29) .. controls (6.95,-1.4) and (3.31,-0.3) .. (0,0) .. controls (3.31,0.3) and (6.95,1.4) .. (10.93,3.29)   ;
%Straight Lines [id:da3169259120122798] 
\draw    (363,188) -- (442,188) ;
\draw [shift={(444,188)}, rotate = 180] [color={rgb, 255:red, 0; green, 0; blue, 0 }  ][line width=0.75]    (10.93,-3.29) .. controls (6.95,-1.4) and (3.31,-0.3) .. (0,0) .. controls (3.31,0.3) and (6.95,1.4) .. (10.93,3.29)   ;

% Text Node
\draw (61,142) node [anchor=north west][inner sep=0.75pt]   [align=left] {1};
% Text Node
\draw (191,142) node [anchor=north west][inner sep=0.75pt]   [align=left] {2};
% Text Node
\draw (332,181) node [anchor=north west][inner sep=0.75pt]   [align=left] {3};
% Text Node
\draw (462,181) node [anchor=north west][inner sep=0.75pt]   [align=left] {4};
% Text Node
\draw (395,85) node [anchor=north west][inner sep=0.75pt]   [align=left] {5};

\end{tikzpicture}
\caption{The associated $G(t)$%\textit{cycle graph}
}
\end{minipage}

\end{figure}
\end{example}
We now consider how many ways there are to start with a table $t \in T_{6,n-a-b}$ together with a pairing $P \in \mathcal{P}$ and construct a table $t' \in D_{n,A,B}$ (we will automatically have that the sign of $t$ and the pairing $P$ is preserved). Before giving the entire analysis, we describe the parts of the analysis corresponding to the cycles and paths of $G(t')$ as these are the trickiest parts of the analysis.
\begin{definition}
Define $C_n$ to be the number of tables $t \in D_{n,\emptyset,[n]}$ such that $G(t)$ consists of directed cycles and for each $i \in [n]$, column $i$ contains four $i$.
\end{definition}
\begin{lemma}\label{lem:countingcycletables}
For all $n \geq 2$, $C_{n}=(n-1)(C_{n-1}+15C_{n-2})$ where $C_0 = 1$ and $C_1 = 0$.
\end{lemma}
\begin{proof}
Consider a vertex $i \in [n]$. $G(t)$ contains an edge from $i$ to $j$ for some $j \in [n] \setminus \{i\}$. Note that there are $n-1$ possibilities for $j$. We now have two cases. The first case is that there is an edge from $j$ to a vertex $k \in [n] \setminus \{i,j\}$. In this case, we can remove the vertex $j$ and the edges $(i,j),(j,k)$ and add an edge from $i$ to $k$. The number of possibilities for this case (for a fixed $j$) is $C_{n-1}$. The second case is that there is an edge from $j$ back to $i$, i.e., $i$ and $j$ are in a directed cycle of length $2$. The number of possibilities for this case (for a fixed $j$) is $15C_{n-2}$ as there are $15$ possibilities for which rows of column $i$ contain $j$ and there are $C_{n-2}$ possibilities for the remaining $n-2$ columns. 

Adding these two cases together and summing over all possible $j \in [n] \setminus \{i\}$, we have that $C_n = (n-1)(C_{n-1}+15C_{n-2})$, as needed.
\end{proof}
To handle paths, we first count the number of possible graphs $G(t)$ with a given number of paths to $\cter$ and no cycles with the following lemma.
\begin{lemma}\label{lem:countingpathgraphs}
Let $B' \subseteq [n]$ and take $j = |B'|$. For all $x \in [j]$, there are $\frac{\binom{j-1}{x-1}}{x!}j!$ possible graphs $G(t)$ on the vertices $B' \cup \{\cter\}$ which consist of $x$ disjoint paths to $\cter$ and no cycles.
\end{lemma}
\begin{proof}
We can specify each such graph as follows:
\begin{enumerate}
    \item Choose an ordering for the elements of $B'$. There are $j!$ possibilities for this ordering.
    \item Choose the $x$ paths by putting $x-1$ dividing lines among the elements of $B'$. Since each path must have at least one vertex in $B'$, there are $\binom{j-1}{x-1}$ possibilities for this.
\end{enumerate}
However, if we do this, each graph is counted $x!$ times, one for each possible ordering of the $x$ paths. Thus, the number of such graphs is $\frac{\binom{j-1}{x-1}}{x!}j!$, as needed.
\end{proof}
In order to have a path $(b_1,b_2),(b_2,b_3),\ldots,(b_l,\cter)$ in $G(t)$, a pair of elements from the columns corresponding to $\cter$ must be placed into the column containing four $b_1$. The two displaced $b_1$ must then be placed into the same rows of the column containing four $b_2$. Continuing in this way, we are left with two $b_{l}$ which replace the original pair of elements from $\cter$. Thus, to specify the contents of the columns corresponding to the paths in $G(t)$, it is necessary and sufficient to choose a pair from the columns corresponding to $\cter$ for each path. Note that these pairs must be different as otherwise we would not end up with disjoint paths.

We now give the entire analysis for $D_{n,a,b}$. Given $A \subseteq [n]$ and $B \subseteq [n] \setminus A$, we can compute $D_{n,a,b} = \sum_{t \in D_{n,A,B}}{\sgn(t)|\mathcal{P}_{A,B}(t)|}$ as follows. As before, we take $a = |A|$ and $b = |B|$.
\begin{enumerate}
    \item For each $a \in A$, we choose which column contains six $a$. Similarly, for each $b \in B$, we choose which column contains four $b$. The number of choices for this is $\prod_{j=0}^{a+b-1}{(n-j)}$.
    \item After choosing these columns, we choose a table $t \in T_{6,n-a-b}$ and a pairing $P \in \mathcal{P}(t)$ to fill in the remaining columns. This gives a factor of $P_{n-a-b}$.
    %Note: T_{6,n-a-b} is not quite accurate as it only has numbers up to n - a - b.
    \item We split into cases based on the number of vertices $i$ in $G(t)$ which are contained in cycles. For each $i$, we choose which $\binom{b}{i}$ of the elements in $B$ are contained in cycles. By Lemma \ref{lem:countingcycletables}, once these elements are chosen there are $C_i$ possibilities for the columns containing these elements.
    \item There are now $j = b - i$ elements of $B$ which are contained in paths. We further split into cases based on the number $x$ of paths in $G(t)$. By Lemma \ref{lem:countingpathgraphs}, there are $\frac{\binom{j-1}{x-1}}{x!}j!$ possibilities for what these paths are in $G(t)$.
    
    As discussed above, for each of the $x$ paths we need to choose a different pair in $P$. The number of choices for these pairs is $\prod_{y=0}^{x-1}{(3(n-a-b)-y)}$. Summing all of these possibilities up gives a factor of
    \[
        H_{n,j,a,b}=\sum_{x=1}^{j}{\frac{\binom{j-1}{x-1}}{x!}j!\prod_{y=0}^{x-1}(3(n-a-b)-y)}.
    \]
\end{enumerate}
Putting everything together, we have that 
\[
    D_{n,a,b}=\left(\prod_{j=0}^{a+b-1}{(n-j)}\right)\left(\sum_{i=0}^{b}\binom{b}{i}C_{i}H_{n,b-i,a,b}\right)P_{n-a-b},
\]
as needed.
\end{proof}
\end{proof}

We now simplify the terms in Lemma \ref{main_lemma}.
\begin{proposition}\label{prop:H}
\begin{equation*}
    H_{n,j,a,b} = \frac{(3(n-a-b)+j-1)!}{(3(n-a-b)-1)!}.
\end{equation*}
\end{proposition}
\begin{proof}
Originally,
\[
H_{n,j,a,b} = \sum_{x=1}^{j}{\frac{\binom{j-1}{x-1}}{x!}j!\prod_{y=0}^{x-1}(3(n-a-b)-y)}.
\]
Denote $z = 3(n-a-b)$. For the inner product, we can write
\[
\prod_{y=0}^{x-1}(3(n-a-b)-y) = \frac{z!}{(z-x)!},
\]
so
\[
H_{n,j,a,b} = \sum_{x=1}^{j}\frac{\binom{j-1}{x-1}}{x!}\frac{j!z!}{(z-x)!} = j! \sum_{x=1}^{j}\binom{j-1}{x-1} \binom{z}{x} = j! \binom{z+j-1}{j}.
\]
The last equality is a special case of the \textbf{Chu-Vandermonde Identity}.
\end{proof}
\begin{comment}
\begin{proposition}\label{prop:D}
\end{proposition}
\[
D_{n,a,b} = \frac{n!(n-a-b+2)!(n-a-b+4)!}{48} \sum_{i=0}^{b}\binom{b}{i}\frac{(3n-3a-2b-i-1)!}{(3n-3a-3b-1)!} C_i.
\]
\begin{proof}
By Lemma \ref{pairing}, $P_n = \frac{n!(n+2)!(n+4)!}{48}$ so $
P_{n - a - b} = \frac{(n - a - b)!(n-a-b+2)!(n-a-b+4)!}{48}$. 
By Proposition \ref{prop:H}, 
$H_{n,j,a,b} = \frac{(3(n-a-b)+j-1)!}{(3(n-a-b)-1)!}$ and plugging in $j = b-i$ gives 
\[
H_{n,b-i,a,b} = \frac{(3(n-a-b)+(b-i)-1)!}{(3(n-a-b)-1)!} = \frac{(3n-3a-2b-i-1)!}{(3n-3a-3b-1)!}.
\]
Using these equations and the observation that $\prod_{j=0}^{a+b-1}{(n-j)} = \frac{n!}{(n - a - b)!}$, we have that 
\begin{align*}
D_{n,a,b} &=\left(\prod_{j=0}^{a+b-1}{n-j}\right)\left(\sum_{i=0}^{b}\binom{b}{i}C_{i}H_{n,b-i,a,b}\right)P_{n-a-b} \\
&= \frac{n!(n-a-b+2)!(n-a-b+4)!}{48} \sum_{i=0}^{b}\binom{b}{i}\frac{(3n-3a-2b-i-1)!}{(3n-3a-3b-1)!} C_i,
\end{align*}
as needed.
%In the definition of $D_{n,a,b}$, that is in
%\[
%D_{n,a,b} =\left(\prod_{j=0}^{a+b-1}{n-j}\right)\left(\sum_{i=0}^{b}\binom{b}{i}C_{i}H_{n,b-i,a,b}\right)P_{n-a-b},
%\]
%rewrite the leftmost product and $P_{n-a-b}$ using factorials only and use Proposition \ref{prop:H}.
\end{proof}
\end{comment}
\begin{lemma}\label{LemFlajo}
Let $S_n$ be the set of all permutations of order $n$ and $D_n$ the set of all \textbf{derangements} of the same order. That means, $D_n$ is a subset of those permutations in $S_n$ which have no fixed points. Denote $C(\pi)$ the number of cycles in a permutation $\pi$, then
\begin{equation*}
\sum_{n=0}^\infty \frac{x^n}{n!}\sum_{\pi \in D_n} u^{C(\pi)} = \frac{e^{-ux}}{(1-x)^u}.
\end{equation*}
\end{lemma}
\begin{proof}
For a derivation, see the chapter on Bivariate generating functions in \cite{flajolet2009analytic}.
\end{proof}
\begin{corollary}\label{cor:Cngen}
\[
\sum_{n = 0}^\infty \frac{t^n}{n!} C_n = \frac{e^{-15x}}{(1-x)^{15}}.
\]
\end{corollary}
\begin{proof}
An alternative way how to write $C_n$ is via $C_n = \sum_{\pi \in D_n} {15}^{C(\pi)}$.
\end{proof}
With these simplifications, we can derive an expression for the generating function
\[
F_6(t) = \sum_{n=0}^\infty \frac{t^n}{(n!)^2} f_6(n).
\]
By Lemma \ref{main_lemma},
\[
F_6(t) = \sum_{0 \leq a \leq j \leq n} \frac{t^n}{(n!)^2} \binom{n}{j}\binom{j}{a}(m_6 - 15)^{a} (m_4 - 3)^{(j-a)}D_{n,a,j-a}.
\]
Summing with respect to $b=j-a$ instead of $a$ and observing that \begin{align*}
D_{n,a,b}&=\left(\prod_{k=0}^{a+b-1}{(n-k)}\right)\left(\sum_{i=0}^{b}\binom{b}{i}C_{i}H_{n,b-i,a,b}\right)P_{n-a-b}\\ 
&= \frac{n!}{(n-j)!}\left(\sum_{i=0}^{b}\binom{b}{i}C_{i}H_{n,b-i,j-b,b}\right)P_{n-j}\\
&=n!\left(\sum_{i=0}^{b}\binom{b}{i}C_{i}H_{n,b-i,j-b,b}\right)\frac{(n\!-\!j\!+\!2)!(n\!-\!j\!+\!4)!}{48},
\end{align*}
we have that
\[
F_6(t) = \!\!\!\!\! \sum_{0 \leq i \leq b \leq j \leq n} \frac{t^n}{n!} \binom{n}{j}\binom{j}{b}\binom{b}{i}(m_6 - 15)^{(j-b)} (m_4 - 3)^{b} \frac{(n\!-\!j\!+\!2)!(n\!-\!j\!+\!4)!}{48}H_{n,b-i,j-b,b} C_i.
\]
By Proposition \ref{prop:H}, $H_{n,b-i,j-b,b} = (3 n-3 j+ b-i-1)!/(3 n-3j-1)!$. Using the reparametrization $b = i + s, j = b + r, n = j + q$, where $s,r,q$ goes from $0$ to $\infty$, we get
\[
F_6(t) = \sum_{q=0}^\infty \sum_{r=0}^\infty \sum_{s=0}^\infty \sum_{i=0}^\infty \frac{t^{i+s+r+q}}{q!r!s!i!} (m_6 - 15)^r (m_4 - 3)^{(i+s)} \frac{(q+2)!(q+4)!}{48} \frac{(3q+s-1)!}{(3q-1)!} C_i.
\]
%Separating the parts of this expression depending on 
Grouping the terms to separate the dependence on $r$, $s$, and $i$, we have that $F_6(t)$ equals
\begin{align*}
 \sum_{q=0}^\infty \frac{t^q}{q!} \frac{(q\!+\!2)!(q\!+\!4)!}{48} \left(\sum_{r=0}^\infty \frac{t^r (m_6\!-\!15)^r}{r!}  \right)\! \left( \sum_{s=0}^\infty \frac{t^s}{s!} \frac{(3q\!+\!s\!-\!1)!}{(3q-1)!} (m_4 \!-\! 3)^s \right)\! \left(\sum_{i=0}^\infty \frac{t^i}{i!} (m_4 \!-\! 3)^i  C_i \right).
\end{align*}
Summing all the inner sums (the rightmost using Corollary \ref{cor:Cngen}),
\[
F_6(t) = \sum_{q=0}^\infty \frac{t^q}{q!} \frac{(q+2)!(q+4)!}{48} e^{t(m_6-15)} \frac{1}{(1-t(m_4-3))^{3q}} \frac{e^{-15t(m_4-3)}}{(1-t(m_4-3))^{15}}.
\]

\section{Generalization for arbitrary third moment}
Restating Proposition \ref{prop:tables}, we can write
\begin{equation*}
    f_6(n) = \sum_{t \in T_{6,n}} w(t)\operatorname{sign}(t),
\end{equation*}
where $T_{6,n}$ is the set of all permutation tables of length $n$ with six rows (six-tables) whose columns fall in one of the following categories
\begin{itemize}
    \item 6-columns: six copies of a single number (weight $m_6$)
    \item 4-columns: four copies of one number and two copies of a distinct number (weight $m_4$)
    \item 2-columns: three pairs of distinct numbers (weight $1$)
\end{itemize}
The weight $w(t)$ of the table $t$ is then simply a product of weights of its columns. To avoid ambiguity, we write $f^*_6(n)$ for $f_6(n)$ with $m_3$ being generally nonzero. We then define $T^*_{6,n}$ as the set of all six-tables having the following extra columns
\begin{itemize}
    \item 3-columns: three copies of one number and three copies of a distinct number (weight $m_3^2$)
\end{itemize}
Similarly, it must hold that
\begin{equation*}
    f^*_6(n) = \sum_{t \in T^*_{6,n}} w(t)\operatorname{sgn}(t).
\end{equation*}
\begin{proposition}
\begin{equation*}
f^*_6(n) = \sum_{j = 0}^n \binom{n}{j}^2 f_6(n-j) \, j! m_3^{2j} (-1)^j \sum_{\pi \in D_j} (-10)^{C(\pi)}.
\end{equation*}
\end{proposition}
\begin{proof}
The key is to group the summands according to the 3-columns in $t$. Those columns form a subtable $s$ and the rest of the columns form another, complementary subtable $t'$. The signs of those tables are related as
\begin{equation*}
 \operatorname{sgn}(t) = \operatorname{sgn}(s) \operatorname{sgn}(t').
\end{equation*}
Denote $\left[n\right] = \{1,2,3,\ldots,n\}$. For a given $J \subset \left[n\right]$, we define $T_{6,J}$ a set of all six-tables of length $j = |J|$ composed with numbers in $J$. The set $T_{6,n}$ coincides with $T_{6,\left[n\right]}$. Denote $Q_{6,J}$ as the set of all six-tables composed only from 3-columns of numbers in $J$. We can write our sum, since the selection $J$ does not depend on position in table $t$, as
\begin{equation*}
    f^*_6(n) = \sum_{J \subset \left[n\right]} \binom{n}{j} \sum_{t' \in T_{6,\left[n\right]/J}} w(t)\operatorname{sgn}(t) \sum_{s \in Q_{6,J}} w(s)\operatorname{sgn}(s).
\end{equation*}
No matter which numbers $J$ are selected, as long as we select the same amount of them, the contribution is the same. Hence,
\begin{equation*}
    f^*_6(n) = \sum_{j = 0}^n \binom{n}{j}^2 \sum_{t' \in T_{6,n-j}} w(t)\operatorname{sgn}(t) \sum_{s \in Q_{6,j}} w(s)\operatorname{sgn}(s),
\end{equation*}
where $Q_{6,j} = Q_{6,\left[j\right]}$. The first inner sum is simply $f_6(n-j)$. For the second inner sum, by symmetry, we can fix the first permutation in $s$ to be identity. Upon noticing also that $w(s) = m_3^{2j}$, we get
\begin{equation*}
\sum_{s \in Q_{6,j}} w(s)\operatorname{sgn}(s) = j! m_3^{2j} \sum_{\substack{s \in Q_{6,j} \\ s_1 = \mathrm{id}}} \operatorname{sgn}(s).
\end{equation*}
We group the summands according to the following permutation structure: Let $b$ be a number in the first row of a given column of table $s$. Since it is a 3-column, we denote the other number in the column as $b'$. We construct a permutation $\pi(s)$ to a given table $s$ as composed from all those pairs $b \rightarrow b'$. Then
\begin{equation*}
\operatorname{sgn}(s) = \operatorname{sign}(\pi(s)) = (-1)^{j-C(\pi(s))}.
\end{equation*}
Note that since $b$ and $b'$ are always different, the set off all $\pi(s)$ corresponds to the set $D_j$ of all derangements. Since there are $10$ possibilities how to arrange the leftover 5 numbers in the 3-columns corresponding to a given cycle of $\pi(s)$, we get
\begin{equation*}
\sum_{s \in Q_{6,j}} w(s)\operatorname{sgn}(s) = j! m_3^{2j} (-1)^j \sum_{\pi \in D_j} (-1)^{C(\pi)} {10}^{C(\pi)}.
\end{equation*}
and thus, all together
\begin{equation*}
    f^*_6(n) = \sum_{j = 0}^n \binom{n}{j}^2 f_6(n-j) \, j! m_3^{2j} (-1)^j \sum_{\pi \in D_j} (-10)^{C(\pi)}.
\end{equation*}
\end{proof}
\begin{corollary}
\begin{equation*}
F_6^*(t) = (1+m_3^2t)^{10} e^{-10m_3^2t} \, F_6(t).
\end{equation*}
\end{corollary}
\begin{proof}
In terms of generating functions,
\begin{equation*}
\begin{split}
    F_6^*(t) & = \sum_{n=0}^\infty \frac{t^n}{n!^2} f^*_6(n) = \sum_{n=0}^\infty\sum_{j = 0}^n \frac{t^{n-j}}{(n-j)!^2} f_6(n-j) \frac{(-m_3^2t)^j}{j!} \sum_{\pi \in D_j} (-10)^{C(\pi)} \\
    & = F_6(t) \sum_{j=0}^\infty\frac{(-m_3^2t)^j}{j!} \sum_{\pi \in D_j} (-10)^{C(\pi)} = F_6(t) \frac{e^{-10m_3^2t}}{(1+m_3^2t)^{-10}}.
\end{split}
\end{equation*}
The final equality is a special case of Lemma \ref{LemFlajo}. Theorem \ref{m3nonzero_thm} follows.
\end{proof}

\section{Asymptotics}
The proof relies directly on the calculus developed by Borinsky \cite{borinsky2018}, enabling us to extract the asymptotic behaviour of coefficients from their factorially divergent generating function. We use the following result from Borinsky \cite{borinsky2018}:
\begin{definition}
We say a formal power series $f(t) = \sum_{n \geq 0} f_n t^n$ is factorially divergent of type $(\alpha,\beta)$, if $f_n \sim \sum_{k=0}^R c_k \alpha^{n+\beta - k} \Gamma(n+\beta-k) $ as $n \to \infty$ for any fixed $R$ integer. We also define an operator $\mathcal{A}^\alpha_\beta$ acting of $f(t)$ such that $(\mathcal{A}^\alpha_\beta f)(t) = \sum_{k\geq 0} c_k t^k$. If moreover $f(t)$ is analytic at $0$, then $(\mathcal{A}^\alpha_\beta f)(t) = 0$.
\end{definition}
\begin{lemma}\label{lem:borinsky}
Let $f(t)$ and $g(t)$ be two factorially divergent power series of type $(\alpha,\beta)$, then
\begin{align*}
(\mathcal{A}^\alpha_\beta (fg))(t) & =  (\mathcal{A}^\alpha_\beta f)(t) g(t) + f(t) (\mathcal{A}^\alpha_\beta g)(t),\\
(\mathcal{A}^\alpha_\beta (f \circ g) )(t) & = f'(g(t)) (\mathcal{A}^\alpha_\beta g)(t) + \left(\frac{t}{g(t)}\right)^\beta e^{\frac{\frac{1}{t}-\frac{1}{g(t)}}{\alpha}} (\mathcal{A}^\alpha_\beta f )(g(t)),
\end{align*}
where the second equality holds when $g(t) = 1 + t + O(t^2)$.
\end{lemma}

Recall Theorem \ref{m3nonzero_thm}, which states
\begin{align*}
F_6(t) =\left(1+m_3^2t\right)\!{}^{10}\, \frac{e^{t \left(m_6-10 m_3^2-15m_4 + 30\right)}}{48 \left(1+3t-m_4t\right)^{15}} \sum _{i=0}^{\infty}
   \frac{(1+i) (2+i) (4+i)!t^i}{\left(1+3t-m_4t\right){}^{3 i}}.
\end{align*}
Hence, we can write $F_6(t) = h(t)f(g(t))$, where
\begin{align*}
f(t) &= \sum_{i=0}^\infty (1+i)(2+i)(4+i)! t^i,\\ 
g(t) &= \frac{t}{(1+3t-m_4t)^3}, \qquad h(t) = \left(1+m_3^2t\right)\!{}^{10}\, \frac{e^{t \left(m_6-10 m_3^2-15m_4 + 30\right)}}{48 \left(1+3t-m_4t\right)^{15}} \\
\end{align*}
are factorially divergent of type $(1,7)$ since
\[
(1+i)(2+i)(4+i)! = \Gamma(i+7)-8\Gamma(i+6)+12\Gamma(i+5)
\]
and $g(t)$ and $h(t)$ are analytic. Thus, by Lemma \ref{lem:borinsky},
\[
(\mathcal{A}^1_7 F_6)(t) = h(t)\left(\frac{t}{g(t)}\right)^7 e^{\frac{1}{t}-\frac{1}{g(t)}} (\mathcal{A}^1_7 f )(g(t)) = h(t) \left(\frac{t}{g(t)}\right)^7 e^{\frac{1}{t}-\frac{1}{g(t)}} (1-8 g(t) + 12 g^2(t)).
\]
Apart from a factor $(n!)^2 e^{3(m_4-3)}/48$, this is our function $C(t)$ from the original statement of Theorem \ref{asym_thm}.

For $\Omega = \{-1,1\}$, the asymptotic expression 
\[
f_6(n) \sim \frac{(n!)^3}{48
   e^6}\left(n^6+29 n^5+335 n^4+\frac{5861 n^3}{3}+\frac{17944 n^2}{3}+\frac{44036 n}{5}+\frac{167536}{45}-\frac{210176}{63 n}\right)
\]
gives an excellent approximation to $f_6(n)$ for $n \geq 10$. The following figure shows the ratio of this asymptotic expression to the actual value of $f_6(n)$ for $n$ up to $20$.

\begin{figure}[h]
\centering
\includegraphics[scale=1.0]{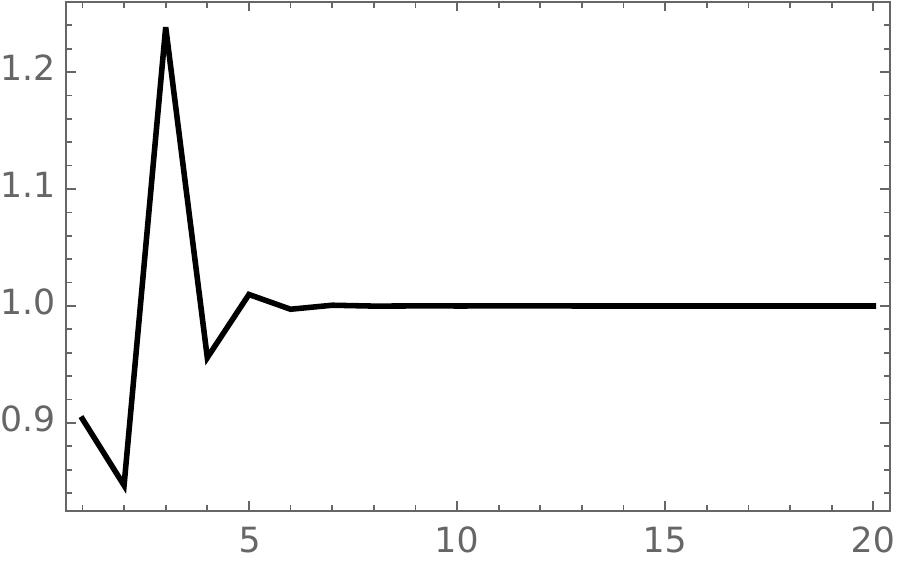}
\caption{The ratio between the asymptotic expression and $f_6(n)$ for $\Omega = \{-1,1\}$}
\end{figure}

\newpage
\begin{appendix}
\section{Direct proof of Lemma \ref{pairing}}\label{sec:directproof}
\begin{no-lemma} (\textit{Restatement of Lemma \ref{pairing}}).
For all $n \in \mathbb{N}$, $P_{n}=n(n+2)(n+4)P_{n-1}$ where $P_0=1$.
\end{no-lemma}
\begin{proof}
We recursively compute $P_{n} = \sum_{t \in T_{k,n}}{\sgn(t)|\mathcal{P}(t)|}$ based on where the six $n$ are located in $t$. 

We can count the cases where all of the $n$ are in a 6-column as follows. Given a table $t \in T_{k,n-1}$ and a pairing $P \in \mathcal{P}(t)$, we can obtain a table $t' \in T_{k,n}$ and a pairing $P' \in \mathcal{P}(t)$ by choosing a location for the 6-column, choosing a pairing for this column, and using $t$ and $P$ to fill in the remainder of $t'$ and $P'$. There are $n$ possible places for the 6-column, it has $15$ possible pairings, and $\sgn(t') = \sgn(t)$, so this gives a contribution of $15nP_{n-1}$. 

We can count the cases where four of the $n$ are in a 4-column and two of the $n$ appear in a different column as follows. Given a table $t \in T_{k,n-1}$ and a pairing $P \in \mathcal{P}(t)$, we can obtain a table $t' \in T_{k,n}$ and a pairing $P' \in \mathcal{P}(t)$ with the following steps:
\begin{enumerate}
    \item Choose which column will be the 4-column containing four of the $n$. We initially put all six $n$ in this column.
    \item Fill in the remaining columns using $t$ and $P$.
    \item Choose one of the $3(n-1)$ pairs in $P$ and swap two of the $n$ with this pair.
    \item Choose a pairing for the remaining four $n$.
\end{enumerate}
There are $n$ possible places for the 4-column containing four of the $n$, there are $3(n-1)$ pairs in $P$ which can be swapped with two of the $n$, there are $3$ different pairings for the remaining four $n$, and $\sgn(t') = \sgn(t)$, so this gives a contribution of $3*3*n(n-1)*P_{n-1} = 9n(n-1)P_{n-1}$.

%If they form a $4-column$, then $n$s appear in two columns, which has $n(n-1)$ possibilities in total, and for the column with 4 $n$ elements, it has $3$ possible pairings and we need to select one of three pairings to place the other 2 $n$ elements. Therefore we have $3*3*n(n-1)*P_{n-1}$ possible pairings. \\
% \textcolor{red}{I think the argument here is considerably more subtle and involved. We'll want to look through our old emails.}\\

The trickiest case to analyze is the case when the six $n$ are split into three different columns. The idea for this case is that there is a correspondence between sets of $2$ columns containing pairs of the elements $a,b,c,d,e,f$ and sets of $3$ columns containing pairs of the elements $a,b,c,d,e,f$ where each column also contains a pair of $n$. This correspondence is highly non-trivial and relies on the signs of the permutations.
\begin{definition}
Let $S_1,S_2,S_3,S_4,S_5,S_6$ be six sets such that each set $S_{i}$ contains two of the elements $\{a,b,c,d,e,f\}$ and each element in $\{a,b,c,d,e,f\}$ is contained in two of the sets $S_1,S_2,S_3,S_4,S_5,S_6$.

We define $T_2(S_1,S_2,S_3,S_4,S_5,S_6)$ to be the set of $6 \times 2$ tables $t$ such that the ith row contains the elements in $S_i$ and each element appears an even number of times in each column. Similarly, we define $T_3(S_1,S_2,S_3,S_4,S_5,S_6)$ to be the set of $6 \times 3$ tables $t$ such that the ith row contains the elements in $S_i \cup \{n\}$ and each element appears an even number of times in each column.

For each $t \in T_2(S_1,S_2,S_3,S_4,S_5,S_6)$, we define $\sgn(t)$ to be the product of the signs of the rows of $t$ where row $i$ of $t$ has sign $1$ if the elements of $S_i$ appear in order and sign $-1$ if the elements of $S_i$ appear out of order. Similarly, for each $t \in T_3(S_1,S_2,S_3,S_4,S_5,S_6)$, we define $\sgn(t)$ to be the product of the signs of the rows of $t$ where row $i$ of $t$ has sign $1$ if it takes an even number of swaps to transform it into $S_i \cup \{n\}$ and $-1$ if it takes an odd number of swaps to transform it into $S_i \cup \{n\}$.
\end{definition}
\begin{lemma}\label{lem:correspondence}
For all possible $S_1,S_2,S_3,S_4,S_5,S_6$, 
\[
\sum_{t \in T_3(S_1,S_2,S_3,S_4,S_5,S_6)}{\sgn(t)} = 6\sum_{t \in T_2(S_1,S_2,S_3,S_4,S_5,S_6)}{\sgn(t)}.
\]
\end{lemma}
\begin{corollary} For all $n \in \mathbb{N}$,
\[
\sum_{t \in T_{6,n}: n \text{ appears in } 3 \text{ different columns}}{\sgn(t)|\mathcal{P}(t)|} = n(n-1)(n-2)P_{n-1}.
\]
\end{corollary}
\begin{proof}
Recall that
\[
\sum_{t \in T_{6,n-1}}{\sgn(t)|\mathcal{P}(t)|} = P_{n-1}.
\]
We now apply Lemma \ref{lem:correspondence} to the first two columns of the pairs $(t,P)$ where $t \in T_{6,n-1}$ and $P \in \mathcal{P}(t)$. To do this, we use $P$ to relabel the elements in the first two columns as $a,b,c,d,e,f$. One way to do this is as follows. We go through the rows one by one and assign the next unused label(s) to the element(s) which whose pair has not yet appeared. If there are two such elements, we assign the first unused label to the lower element and the next unused label to the higher element. If both elements are the same, we assign the first unused label to the column where the pair of this element appears first. If there is still a tie, we assign the same label to both elements and skip the next label. Lemma \ref{lem:correspondence} still holds in this case as having $S_{i} = S_{j} = \{a,a\}$ instead of $S_{i} = S_{j} = \{a,b\}$ divides both sides by $2$.

After doing this relabeling, for each $i \in [6]$, we take $S_i$ to be the first two elements in row $i$. Applying Lemma \ref{lem:correspondence}, we obtain tables $t'$ and pairings $P'$ by taking $P'$ to be the unique pairing for each column and inverting the labeling of the elements in the first two columns of $t$ by $\{a,b,c,d,e,f\}$. This implies that whenever $n \geq 3$,
\[
\sum_{t \in T_{6,n}: n \text{ appears in the first three columns}}{\sgn(t)|\mathcal{P}(t)|} = 6\sum_{t \in T_{6,n-1}}{\sgn(t)|\mathcal{P}(t)|} = 6P_{n-1}.
\]
There are $\binom{n}{3} = \frac{n(n-1)n-2)}{6}$ possibilities for which $3$ columns contain $n$ so we have that 
\[
\sum_{t \in T_{6,n}: n \text{ appears in } 3 \text{ different columns}}{\sgn(t)|\mathcal{P}(t)|} = n(n-1)(n-2)P_{n-1},
\]
as needed.
\end{proof}

Summing these three cases up, we have
\begin{align*}
  P_n&=15nP_{n-1}+9n(n-1)P_{n-1}+n(n-1)(n-2)P_{n-1}\\  
  &=(n^3 + 6 n^2 + 8 n)P_{n-1} = n(n+2)(n+4)P_{n-1}.
\end{align*}
\end{proof}
We now prove Lemma \ref{lem:correspondence}.
\begin{proof}[Proof of Lemma  \ref{lem:correspondence}]
Up to permutations of the rows and $\{a,b,c,d,e,f\}$, we have the following four cases for $S_1,S_2,S_3,S_4,S_5,S_6$:
\begin{enumerate}
    \item $S_1 = S_2 = \{a,b\}$, $S_3 = S_4 = \{c,d\}$, and $S_5 = S_6 = \{e,f\}$.
    \item $S_1 = S_2 = \{a,b\}$, $S_3 = \{c,d\}$, $S_4 = \{c,e\}$, $S_5 = \{d,f\}$, and $S_6 = \{e,f\}$.
    \item $S_1 = \{a,b\}$, $S_2 = \{a,c\}$, $S_3 = \{b,d\}$, $S_4 = \{d,e\}$, $S_5 = \{c,f\}$, and $S_6 = \{e,f\}$.
    \item $S_1 = \{a,b\}$, $S_2 = \{a,c\}$, $S_3 = \{b,c\}$, $S_4 = \{d,e\}$, $S_5 = \{d,f\}$, and $S_6 = \{e,f\}$.
\end{enumerate}
We can see that these are the only possibilities as follows. If we construct a multi-graph where the vertices are $\{a,b,c,d,e,f\}$ and the edges are $\{S_1,S_2,S_3,S_4,S_5,S_6\}$ then in this multi-graph, every vertex will have degree $2$. 
\begin{enumerate}
    \item If there is a cycle of length $2$ then for the remaining $4$ vertices, we will either have two more cycles of length $2$ or a cycle of length $4$. This gives cases 1 and 2.
    \item If there is a cycle of length $3$ then there must be another cycle of length $3$ on the remaining vertices. This gives case 4.
    \item If there are no cycles of length $2$ or $3$ then we must have a cycle of length $6$. This gives case 3.
\end{enumerate}

For the first three cases, $T_2(S_1,S_2,S_3,S_4,S_5,S_6)$ is nonempty as shown by the examples below. For the fourth case, $T_2(S_1,S_2,S_3,S_4,S_5,S_6)$ is empty.
\[
\begin{Bmatrix}
    a&b\\a&b\\c & d\\c & d\\e & f\\e & f \end{Bmatrix}, 
\begin{Bmatrix}
    a&b\\a&b\\c & d\\c & e\\f & d\\f & e \end{Bmatrix}, 
\begin{Bmatrix}
    a&b \\a&c \\d & b\\d & e\\f & c\\f& e
\end{Bmatrix}
\]
For all four cases, $T_3(S_1,S_2,S_3,S_4,S_5,S_6)$ is nonempty as shown by the examples below.
\[
\begin{Bmatrix}
    a&b & n\\a&b & n\\c&n & d\\c&n & d\\n&e & f\\n&e & f\end{Bmatrix}, 
\begin{Bmatrix}
    a&b & n\\a&b & n\\c&n & d\\c&n & e\\n&f & d\\n&f& e
\end{Bmatrix}, 
\begin{Bmatrix}
    a&b & n\\a&c & n\\n&b & d\\e&n & d\\n&c & f\\e&n& f
\end{Bmatrix},
\begin{Bmatrix}
    a&b & n\\a&n & c\\n&b & c\\d&e & n\\d&n & f\\n&e& f
\end{Bmatrix}
\]
We now show that for each of the four cases, 
\[
\sum_{t \in T_3(S_1,S_2,S_3,S_4,S_5,S_6)}{\sgn(t)} = 6\sum_{t \in T_2(S_1,S_2,S_3,S_4,S_5,S_6)}{\sgn(t)}.
\]
\begin{enumerate}
    \item For the first case, $|T_2(S_1,S_2,S_3,S_4,S_5,S_6)| = 8$ as we can choose the order of $\{a,b\}$ in row $1$, the order of $\{c,d\}$ in row $3$, and the order of $\{e,f\}$ in row $5$. All $t \in T_2(S_1,S_2,S_3,S_4,S_5,S_6)$ have positive sign as rows 2, 4, and 6 must be the same as rows 1, 3, and 5. Thus, $\sum_{t \in T_2(S_1,S_2,S_3,S_4,S_5,S_6)}{\sgn(t)} = 8$.
    
    To analyze $T_3(S_1,S_2,S_3,S_4,S_5,S_6)$, observe that there are $6$ choices for the positions of the $n$ in rows $1$, $3$, and $5$ and we can again choose the order of $\{a,b\}$ in row $1$, the order of $\{c,d\}$ in row $3$, and the order of $\{e,f\}$ in row $5$. Thus, $|T_3(S_1,S_2,S_3,S_4,S_5,S_6)| = 48$. All $t \in T_3(S_1,S_2,S_3,S_4,S_5,S_6)$ have positive sign as rows 2, 4, and 6 must be the same as rows 1, 3, and 5 so we have that $\sum_{t \in T_3(S_1,S_2,S_3,S_4,S_5,S_6)}{\sgn(t)} = 48$.
    \item For the second case, $|T_2(S_1,S_2,S_3,S_4,S_5,S_6)| = 4$ as we can choose the order of $\{a,b\}$ in row $1$ and the order of $\{c,d\}$ in row $3$ and this uniquely determines the rest of the table. It can be checked that all $t \in T_2(S_1,S_2,S_3,S_4,S_5,S_6)$ have positive sign so we have that $\sum_{t \in T_2(S_1,S_2,S_3,S_4,S_5,S_6)}{\sgn(t)} = 4$.
    
    To analyze $T_3(S_1,S_2,S_3,S_4,S_5,S_6)$, observe that there are $6$ choices for the order of $\{a,b,n\}$ in row $1$. Once this order is chosen, there are two choices for the position of the $n$ in row $3$ and two choices for the order of $\{c,d\}$ in row $3$. It can be checked that this uniquely determines the rest of the table and all $t \in T_3(S_1,S_2,S_3,S_4,S_5,S_6)$ have positive sign so we have that $\sum_{t \in T_3(S_1,S_2,S_3,S_4,S_5,S_6)}{\sgn(t)} = 24$.
    \item For the third case, $|T_2(S_1,S_2,S_3,S_4,S_5,S_6)| = 2$ as we can choose the order of $\{a,b\}$ in row $1$ and this uniquely determines the rest of the table. Here both $t \in T_2(S_1,S_2,S_3,S_4,S_5,S_6)$ have negative sign so we have that $\sum_{t \in T_2(S_1,S_2,S_3,S_4,S_5,S_6)}{\sgn(t)} = -2$.
    
    For $T_3(S_1,S_2,S_3,S_4,S_5,S_6)$, there are $6$ choices for the order of $\{a,b,n\}$ in row $1$. When row $1$ is $a,b,n$, we have the following four tables:
    \[
    \begin{Bmatrix}
    a&b & n\\a&c & n\\n&b & d\\e&n & d\\n&c & f\\e&n& f
    \end{Bmatrix},
    \begin{Bmatrix}
    a&b & n\\a&n & c\\d&b & n\\d&n & e\\n&f & c\\n&f& e
    \end{Bmatrix},
    \begin{Bmatrix}
    a&b & n\\a&n & c\\n&b & d\\e&n & d\\n&f & c\\e&f& n
    \end{Bmatrix},
    \begin{Bmatrix}
    a&b & n\\a&n & c\\n&b & d\\n&e & d\\f&n & c\\f&e& n
    \end{Bmatrix}
    \]
    Of these tables, the first, second, and fourth table have negative sign while the third table has positive sign so the net contribution is $-2$. Multiplying this by $6$, we have that 
    \[
    \sum_{t \in T_3(S_1,S_2,S_3,S_4,S_5,S_6)}{\sgn(t)} = -12.
    \]
    \item For the fourth case, $T_2(S_1,S_2,S_3,S_4,S_5,S_6)$ is empty because each column can only contain one of $\{a,b,c\}$ and one of $\{b,c,d\}$.
    
    To analyze $T_3(S_1,S_2,S_3,S_4,S_5,S_6)$, observe that we can choose the order of $\{a,b,n\}$ in row 1 and the order of $\{d,e,n\}$ in row 4 and this uniquely determines the rest of the table. The sign of each table will be the product of the sign for row $1$ and the sign for row $4$, so we have the same number of tables with positive and negative sign and thus 
    $\sum_{t \in T_2(S_1,S_2,S_3,S_4,S_5,S_6)}{\sgn(t)} = \sum_{t \in T_3(S_1,S_2,S_3,S_4,S_5,S_6)}{\sgn(t)} = 0$.
\end{enumerate}
\end{proof}
\end{appendix}
\end{document}